\documentclass{article}[11pt]

\usepackage{graphicx,tikz,color}
\usepackage{amsmath,amsfonts,amssymb,latexsym}

\newtheorem{theorem}{Theorem}[section]
\newtheorem{proposition}[theorem]{Proposition}

\newtheorem{conjecture}[theorem]{Conjecture}
\newtheorem{corollary}[theorem]{Corollary}

\newtheorem{lemma}[theorem]{Lemma}

\newcommand{\qed}{\hfill ~$\square$\bigskip}
\newcommand{\proof}{\noindent{\bf Proof.} }

\newcommand{\cp}{\,\square\,}

\newcommand{\ggr}{\gamma_{gr}}

\begin{document}

\title{Dominating sequences in grid-like and toroidal graphs}

\author{
Bo\v stjan Bre\v sar $^{a,b}$ \and Csilla Bujt{\'a}s $^{c,d}$ \and
Tanja Gologranc $^{a,b}$ \and Sandi Klav\v zar $^{e,a,b}$ \and Ga\v sper
Ko\v smrlj $^{b}$ \and Bal{\'a}zs Patk{\'o}s $^{d}$ \and Zsolt Tuza
$^{c,d}$ \and M{\'a}t{\'e} Vizer $^{d}$ }

\maketitle

\begin{center}
$^a$ Faculty of Natural Sciences and Mathematics, University of Maribor, Slovenia\\
\medskip

$^b$ Institute of Mathematics, Physics and Mechanics, Ljubljana, Slovenia\\
\medskip

$^c$ Department of Computer Science and Systems Technology, University of Pannonia, Veszpr\'em, Hungary\\
\medskip

$^d$ Alfr\'ed R\'enyi Institute of Mathematics, Hungarian Academy of
Sciences, Budapest, Hungary
\medskip

$^e$ Faculty of Mathematics and Physics, University of Ljubljana, Slovenia\\
\medskip

\end{center}

\begin{abstract}
A longest sequence $S$ of distinct vertices of a graph $G$ such that each vertex of $S$ dominates some vertex that is not dominated by its preceding vertices, is called a Grundy dominating sequence; the length of $S$ is the Grundy domination number of $G$. In this paper we study the Grundy domination number in the four standard graph products: the Cartesian, the lexicographic, the direct, and the strong product. For each of the products  we present a lower bound for the Grundy domination number which turns out to be exact for the lexicographic product and is conjectured to be exact for the strong product. In most of the cases exact Grundy domination numbers are determined for products of paths and/or cycles. 
\end{abstract}


\noindent
{\bf AMS Subj.\ Class.\ (2010)}: 05C69, 05C76

\section{Introduction}

If $G$ is a graph, then a sequence $S=(v_1,\ldots,v_k)$ of distinct vertices of $G$ is called a {\em legal (closed neigborhood) sequence} if, for each $i$,
$$N[v_i] \setminus \bigcup_{j=1}^{i-1}N[v_j]\not=\emptyset\,.$$
In words, for any $i$ the vertex $v_i$ dominates at least one vertex not dominated by $v_1, \ldots, v_{i-1}$. If the set of vertices from a legal sequence $S$ forms a dominating set of $G$, then $S$ is called a {\em dominating sequence} of $G$. Clearly, the length of a dominating sequence is bounded from below by the domination number $\gamma(G)$ of a graph $G$. The maximum length of a dominating sequence in $G$ is called the {\em Grundy domination number} of $G$ and denoted by $\gamma_{gr}(G)$. The corresponding sequence is called a {\em Grundy dominating sequence} of $G$ or $\gamma_{gr}$-sequence of $G$.

These concepts were introduced in~\cite{bgmrr-2014}, where the Grundy domination number
was established for graphs of some well-known classes (such as split graphs and cographs),
 and a linear time algorithm to determine this number in an arbitrary tree was presented.
 It was also shown that the decision version of the problem is NP-complete, even when restricted to chordal graphs.
 Motivated by the question, when a Grundy dominating sequence produces a minimum dominating set
  (and hence any legal sequence is of the same length), the classes of graphs with domination numbers up to $3$ having this property have been characterized. The concept was further studied in~\cite{bgk2016+}, where exact formulas for Grundy domination numbers of Sierpi\' nski graphs were proven, and a linear algorithm for determining these numbers in arbitrary interval graphs was presented.

It is natural to study graph invariants  on graph products,
especially on the four standard ones: the Cartesian, the
lexicographic, the direct and the strong product~\cite{hik-2011}.
Among different reasons for the popularity of graph products, we
emphasize the fact that several intriguing questions, related to
products of graphs, gave new insights and new developments in
studies of the involved invariants. Let us mention the still open,
famous Vizing's conjecture on the domination number of the Cartesian
product of graphs, which was posed in the 1960's~\cite{V1963}, and
initiated the introduction of several new concepts and methods,
cf.~\cite{bresar-2012}. The domination number of the Cartesian
product of paths was completely  determined only in
2011~\cite{gprt-2011}, after a long period in which various attempts
produced different approaches to graph domination problems.
Cartesian products of paths by cycles were investigated
in~\cite{pavlic-2013} and products of cycles in~\cite{klavzar-1995}.

 With respect to other
standard products, an emphasis was given on $r$-perfect codes.
Recall that a set $S$ of vertices of a graph $G$ is an $r$-perfect
code if vertices from $S$ are pairwise at distance at least $2r+1$
and every vertex of $G$ is at distance at most at $r$ from some
vertex from $S$. It is easy to observe that if a graph admits a
$1$-perfect code, then its size is the domination number of the
graph. In~\cite{klavzar-2006,zerovnik-2008} perfect codes in direct
products of cycles were characterized, while in~\cite{abay-2009}
and~\cite{taylor-2009} perfect codes were investigated in strong and
lexicographic products, respectively. In particular, the strong
product of graphs has a perfect $r$-code if and only if each factor
has a perfect $r$-code.
  For a uniform treatment of domination (and
other invariants) on graph products see~\cite{nowakowski-1996}.

In this paper we focus on dominating sequences in the four standard
graph products.  While we could find some general results or bounds
for Grundy domination numbers in all four graph products, we
nevertheless focused mainly on the products of paths and/or cycles.
(As noted above, obtaining the exact formulas for the standard
domination number of these relatively simple classes of graphs has
turned out to be rather non-trivial.) For each of the four products
we consider the Grundy domination number of {\em grids} (the
products of paths), {\em cylinders} (the products of paths by
cycles), and {\em tori} (the products of cycles) and proceed as
follows. In the next section we introduce concepts and terminology
needed, and present two general upper bounds on the Grundy
domination number which will be applied later and could be of
independent interest. In particular, the Grundy domination number is bounded by the edge clique number. In the subsequent four sections we
respectively consider the Grundy domination number of the four
standard graph products. 

In Section~\ref{sec:Cartesian} we first prove that the Grundy domination number of two graphs is bounded from below by the product of the Grundy domination number of one factor and the order of the other factor. We then show that this bound is attained in all grids, cylinders and tori, with the sole exception of the Cartesian product of two cycles of the same odd length. In addition, certain isoperimetric inequalities enable us to determine the Grundy domination numbers of certain multiple Cartesian products. In the case of the lexicographic product (see Section~\ref{sec:lex}), we give a formula for the Grundy domination number of the product of two arbitrary graphs as a function of all dominating sequences of the first factor. This yields exact values of the Grundy domination numbers of products of paths and/or cycles. In the main part of Section~\ref{sec:direct} we consider lower bounds on the Grundy domination number of direct products of graphs. We also give an upper bound for the products of two paths. These results yield an exact value in the case when a shortest path is of even order. The main theme of the section on the strong product (Section~\ref{sec:strong}) is the following conjecture posed for arbitrary graphs $G$ and $H$:  
$$\gamma_{gr}(G\boxtimes H)= \gamma_{gr}(G)\gamma_{gr}(H)\,.$$
While $\gamma_{gr}(G\boxtimes H)\ge \gamma_{gr}(G)\gamma_{gr}(H)$ is easily seen to be true, we prove the reversed inequality if one of the factors is a caterpillar.

\section{Preliminaries} \label{sec:prelim}

If $S=(v_1,\ldots,v_k)$ is a sequence of distinct vertices of a graph $G$, then the corresponding set $\{v_1,\ldots,v_k\}$ will be denoted by $\widehat{S}$. The initial segment $(v_1,\dots,v_i)$ of $S$ will be denoted by $S_i$. Each vertex $u\in N[v_i] \setminus \bigcup_{j=1}^{i-1}N[v_j]$
is called a {\em private neighbor} of $v_i$ with respect to $\{v_1,\ldots,v_{i}\}$. We will also use a more suggestive term by saying that $v_i$ {\em footprints} the vertices from $N[v_i] \setminus \bigcup_{j=1}^{i-1}N[v_j]$, and that $v_i$ is the {\em footprinter} of any $u\in N[v_i] \setminus \bigcup_{j=1}^{i-1}N[v_j]$. For a dominating sequence $S$ any vertex in $V(G)$ has a unique footprinter in $\widehat{S}$. Thus the function $f_S:V(G)\to \widehat{S}$ that maps each vertex to its footprinter is well defined.

Arising from the total domination number of graph, a related invariant was introduced recently in~\cite{bhr-2016}, which is defined on graphs $G$ with no isolated vertices.
A sequence $S=(v_1,\ldots,v_k)$ of distinct vertices of $G$, is called a {\em legal open neighborhood sequence} if $N(v_i) \setminus \bigcup_{j=1}^{i-1}N(v_j) \ne\emptyset$, holds for every $i\in\{2,\ldots,k\}$.
If, in addition, $\widehat{S}$ is a total dominating set of $G$, then we call $S$ a \emph{total dominating sequence} of $G$. (The meaning of {\em footprinting} in the context of total dominating sequences is analogous as above, and should be clear.)
The maximum length of a total dominating sequence in $G$ is called the \emph{Grundy total domination number} of $G$ and is denoted by $\ggr^t(G)$. The corresponding sequence is called a \emph{Grundy total dominating sequence} of $G$.

Recall that for all of the standard graph products, the vertex set of the product of graphs $G$ and $H$ is equal to $V(G)\times V(H)$,
while their edge-sets are as follows. In the \emph{lexicographic product} $G\circ H$ (also denoted by $G[H]$), vertices
$(g_{1},h_{1})$ and $(g_{2},h_{2})$ are adjacent if either $g_{1}g_{2}\in E(G)$ or ($g_{1}=g_{2}$ and $h_{1}h_{2}\in E(H)$).
In the \emph{strong product} $G\boxtimes H$ vertices $(g_{1},h_{1})$ and $(g_{2},h_{2})$ are adjacent whenever
($g_1g_2\in E(G)$ and $h_1=h_2$) or ($g_1=g_2$ and $h_1h_2\in E(H)$) or ($g_1g_2\in E(G)$ and $h_1h_2\in E(H)$).
In the \emph{direct product} $G\times H$ vertices $(g_{1},h_{1})$ and $(g_{2},h_{2})$ are adjacent when
$g_1g_2\in E(G)$ and $h_1h_2\in E(H)$.
Finally, in the \emph{Cartesian product} $G\cp H$ vertices $(g_{1},h_{1})$ and $(g_{2},h_{2})$ are adjacent when
($g_1g_2\in E(G)$ and $h_1=h_2$) or ($g_1=g_2$ and $h_1h_2\in E(H)$).
Hence in general we have $E(G\cp H)\subseteq E(G\boxtimes H)\subseteq E(G\circ H)$, and $E(G\times H)\subseteq E(G\boxtimes H)$,
$E(G\times H)\cap E(G\cp H)=\emptyset$, while $E(G\times H)\cup E(G\cp H)=E(G\boxtimes H)$.
With the exception of the lexicographic product, the standard graph products are commutative. For this and other properties of the standard products see~\cite{hik-2011}.

Let $G$ and $H$ be graphs and $*$  be one of the four graph products
under consideration. For a vertex $h\in V(H)$, we call the set
$G^{h}=\{(g,h)\in V(G * H):\ g\in V(G)\}$ a $G$-\emph{layer} or a
{\em row} of $G * H$. By abuse of notation we will also consider
$G^{h}$ as the corresponding induced subgraph. Clearly $G^{h}$ is
isomorphic to $G$ unless $*$ is the direct product in which case it
is an edgeless graph of order $|V(G)|$. For $g\in V(G)$, the
$H$-\emph{layer} or the {\em column} $^g\!H$ is defined as $^g\!H
=\{(g,h)\in V(G * H):\ h\in V(H)\}$. We may again consider $^g\!H$
as an induced subgraph when appropriate. The map $p_{G}:V(G *
H)\rightarrow V(G)$, $p_{G}(g,h) = g$, is the \emph{projection} onto
$G$ and $p_{H}:V(G * H)\rightarrow V(H)$, $p_{H}(g,h) = h$, is the
\emph{projection} onto $H$. We say that $G * H$ is \emph{nontrivial}
if both factors are graphs on at least two vertices.

We now give two upper bounds on the Grundy domination number of
arbitrary graphs. If $G$ is a graph and $S \subseteq V(G)$, then the
{\it boundary} $\partial{S}$ of $S$ is defined with
$$\partial{S}=\{u \in V(G)\setminus S:\ u \textrm{ has a neighbor in } S\}.$$ 
The first upper bounds reads as follows.

\begin{lemma}
\label{boundary} For a graph $G$ the inequality $\gamma_{gr}(G)\le |V|-k$ holds if and only if every Grundy dominating sequence has an initial segment $S$ such that $|\partial{S}| \ge k$. In particular, if there exists an $m\le |V(G)|$ such that for any $m$-subset $A$ of $V(G)$ we have $|\partial{A} | \ge k$, then $\gamma_{gr}(G)\le |V|-k$ holds.
\end{lemma}

\begin{proof}
If $\gamma_{gr}(G)\le |V|-k$ holds, then for the set $S$ of all the elements of any Grundy dominating sequence we must have $|\partial{S}|\ge k$. The other direction follows from the fact that the size of the boundary of the initial segments is non-decreasing and obviously for the set $S$ of all the elements of any Grundy dominating sequence we have $|S|+|\partial{S}|=|V(G)|$. \qed
\end{proof}

We follow with the next general upper bound, which will be used in
Section~\ref{con:strong}  in the context of strong products of
graphs. Let $G$ be a graph. If ${\cal{C}}=\{ Q_1,\ldots , Q_r\}$ is
a set of cliques of $G$ such that any edge of $G$ is contained in a
clique of $\cal{C}$, then $C$ is called an {\it edge clique cover}
of $G$. The size of a smallest edge clique cover of $G$ is the {\it
edge clique cover number} of $G$ and denoted by $\theta_e(G).$ It is
well-known that $\theta_e(G)$ is equal to the intersection number of
$G$, cf.~\cite{egp-66,mcmc-99}.

\begin{proposition}
\label{prp:edge_cover}
If $G$ is a graph without isolated vertices, then $\ggr(G) \leq \theta_e(G).$
\end{proposition}
\proof Let $\cal{Q}$ be a minimum edge clique cover of $G$ and let
$S=(d_1,\ldots , d_p)$  be a Grundy dominating sequence in $G$,
where $p=\ggr(G).$ We claim that at each term of the sequence $S$ at
least one clique of $\cal{Q}$ becomes completely dominated. Since
$d_i$ is a legal choice, at least one vertex from $N[d_i]$ is not
dominated by $\widehat{S_{i-1}}.$ If $d_i$ footprints itself, let
$x$ be an arbitrary neighbor of $d_i$. (Note that $x$ exists because
$G$ has no isolated vertices.) And if $d_i$ footprints some other
vertex, let $x$ be this vertex. In either of the cases the edge
$d_ix$ lies in a clique $Q \in \cal{Q}$ that is not yet completely
dominated. It follows that  after $d_i$ is selected, all the
vertices of $Q$ are dominated. Consequently, $\cal{Q}$ contains at
least $p$ cliques, that is $\theta_e(G)=|{\cal{Q}}| \geq p=\ggr(G).$
\qed

The bound in Proposition~\ref{prp:edge_cover} is sharp,  as
demonstrated by complete graphs $K_n$ and paths $P_n$. Moreover,
within the class of trees, the above bound  is sharp precisely in
caterpillars (i.e., the trees in which there is no vertex having
more than two non-leaf neighbors). Indeed, note that in bipartite
graphs $\theta_e(G)$ equals the number of edges, and, as mentioned
in~\cite{bgmrr-2014}, in any caterpillar $G$ we have
$\ggr(G)=|V(G)|-1$. On the other hand, \cite[Lemma 2.4]{bgmrr-2014}
implies that if there exists a vertex in a tree with more than two
non-leaf neighbors, then $\ggr(G)<|V(G)|-1$. Hence, caterpillars are
really the only trees that enjoy the equality in the bound from
Proposition~\ref{prp:edge_cover}.

\section{Cartesian product} 
\label{sec:Cartesian}

In this section we determine the Grundy domination number of grids,
cylinders and tori (with respect to the Cartesian product) and use
isoperimetric inequalities to determine the Grundy domination number
of some   products of several paths and cycles.  We begin with the
following general lower bound.

\begin{proposition}
\label{prp:cart}
For any two graphs $G$ and $H$, $$\gamma_{gr}(G\cp H)\ge \max\{\gamma_{gr}(G)|V(H)|,\gamma_{gr}(H)|V(G)|\}.$$
\end{proposition}

\proof
Set $V(H)=\{h_1,\ldots,h_k\}$.
Let $m=\gamma_{gr}(G)$ and let $(d_1,\ldots,d_m)$ be a Grundy dominating sequence in $G$.
Observe that $$((d_1,h_1),\ldots,(d_1,h_k),(d_2,h_1),\ldots,(d_2,h_k),\ldots,(d_m,h_1),\ldots,(d_m,h_k))$$
is a dominating sequence in $G\cp H$. Indeed, $(d_i,h_j)$ footprints $(g_i,h_j)$, where $g_i$ is footprinted by $d_i$ in $G$. Hence $\gamma_{gr}(G\cp H)\ge mk=\gamma_{gr}(G)|V(H)|$.
By reversing the roles of $G$ and $H$ the statement follows. 
\qed

It is natural to ask whether in the inequality in Proposition~\ref{prp:cart} always equality holds. To see that this need not be the case consider the Cartesian product of complete graphs. Indeed,  let $V(K_n)=\{a_1,\ldots , a_n\}$, $V(K_{m})=\{b_1,\ldots , b_m\}$ and $m,n \geq 3$. Then
$$((a_1,b_1),(a_2,b_1),\ldots , (a_{n-1},b_1),(a_1,b_2),(a_1,b_3),\ldots ,(a_1,b_m))$$
is a dominating sequence of length $n+m-2$, which in turn implies that $$\max\{\gamma_{gr}(K_n)|V(K_m)|, \gamma_{gr}(K_m)|V(K_n)|\}=\max\{m,n\} < n+m-2 \leq \ggr(K_n \cp K_m)\,.$$ This example actually shows that the left-hand side of the inequality of Proposition~\ref{prp:cart} can be arbitrary larger that the right-hand side of it.

\begin{theorem}
\label{cartesiangrid}
Let $P_k$ be the path and $C_k$ be the cycle on $k$ vertices. Then we have

\begin{enumerate}
\item[(a)] $\ggr(P_k\cp P_l)=k(l-1)$, if  $2\le k\le l$;
\item[(b)] $\ggr(P_k \cp C_l)=\max \{l(k-1), k(l-2)\}$, if $2 \le k$ and $3 \le l$;
\item[(c)]$ \ggr(C_k\cp C_l)=k(l-2)$, if  $3\le k\le l$ and $(k,l)\neq (2t+1,2t+1)$ for some $1 \le t$;
\item[(d)] $\ggr(C_k\cp C_k)=k(k-2)+1$ if $k$ is odd.
\end{enumerate}
\end{theorem}

\proof
The fact that the left-hand side in \textbf{(a)}, \textbf{(b)} and \textbf{(c)} is at least the right-hand side is a consequence of Proposition $\ref{prp:cart}$ and the facts that $\ggr(P_l)=l-1$ and $\ggr(C_k)=k-2$.

To prove the same for \textbf{(d)} we construct a  dominating sequence of length $k(k-2)+1$ in $C_k \cp C_k$ for odd $k$. We do it in two steps.

Suppose that $k=2t+1$. Let $V(C_k \cp C_k)=\{(x,y) \in \mathbb{Z}^2 : |x|,|y| \le t\}$ and $E(C_k \cp C_k)=\{\{(x,y),(x',y')\}: |x-x'|+|y-y'|=1, \ \textrm{or} \ x=x' \ \textrm{and} \ y=-y'=t, \ \textrm{or} \ y=y' \ \textrm{and} \ x=-x'=t \}$.

Set $$A_1:=\{(x,y) \in V(C_k \cp C_k) : |x-1/2|+|y| \le t-1/2\}$$ with the following ordering: for $(x,y),(x',y') \in A_1,$ $(x,y) \prec_1 (x',y')$ if and only if

\begin{itemize}
\item $|y| < |y'|$, or
\item $|y|=|y'|$ and $y' < 0 < y$, or
\item $y=y'$ and $x < x'$.
\end{itemize}

It is obvious that $\prec_1$ is a total (and so a well-)ordering of  $A_1$. We choose the vertices of $A_1$ into our dominating sequence in the order $\prec_1$. As $(x,y) \in A_1$ will dominate the undominated vertex $(x,y+1)$ if $y \ge 0$ and $(x,y-1)$ if $y < 0$, each element of $A_1$ will be a legal choice in this order.

In the next step we consider the set
$$A_2:=\{(x,y)\in V(C_k \cp C_k) : |x| < t \ \textrm{and} \ |x-1/2|+|y| > t-1/2\}$$ with the following ordering. For $(x,y),(x',y') \in A_2,$ $(x,y) \prec_2 (x',y')$ if and only if
\begin{itemize}
\item $|x| < |x'|$, or
\item $|x|=|x'|$ and $ x > 0 > x'$, or
\item $x=x'$ and $y < y'$.
\end{itemize}

One can easily check that $\prec_2$ is a total (and hence a well-)ordering of $A_2$. We will choose the elements of the dominating sequence after the elements of $A_1$ in the order $\prec_2$. Each element can be chosen into a legal sequence, since $(x,y)\in A_2$ dominates the so far undominated element $(x+1,y)$ if $x > 0$ and $(x-1,y)$ if $x \leq 0$.

We are done since the length of this dominating sequence is $k(k-2)+1$. A dominating sequence of $C_5 \cp C_5$ is constructed in Figure~\ref{fig:CtimesC}.

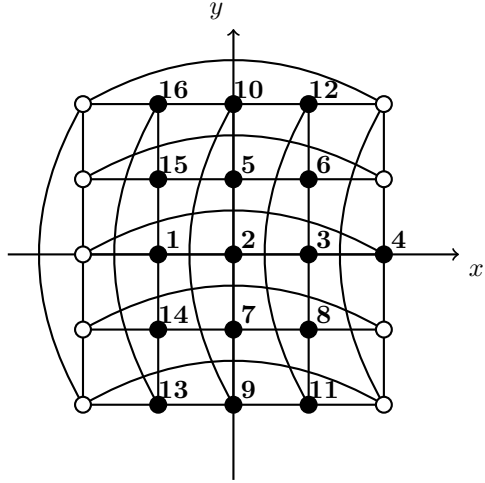
\begin{figure}[ht!]
    \begin{center}
        \begin{tikzpicture}[scale=1.0,style=thick,x=1cm,y=1cm]
        \def\vr{2.5pt} 
        \draw[thick, ->] (-3,0) -- (3,0) node[anchor = north west] {$x$};
    \draw[thick, ->] (0,-3) -- (0,3) node[anchor = south east] {$y$};

    \foreach \x  in {-2,-1,0,1,2}
    {
    \draw (\x,-2) -- (\x,2);
    \draw (\x,-2) to [bend left=30] (\x,2);
    }
    \foreach \y  in {-2,-1,0,1,2}
    {
    \draw (-2,\y) -- (2,\y);
    \draw (-2,\y) to [bend left=30] (2,\y);
    }

        \foreach \x  in {-1,0,1}
    \foreach \y  in {-2,-1,0,1,2}
        {
        \filldraw [fill=black, draw=black,thick] (\x,\y) circle (3pt);
        }
        \filldraw [fill=black, draw=black,thick] (2,0) circle (3pt);

    \foreach \y  in {-2,-1,0,1,2}
        {
        \filldraw [fill=white, draw=black,thick] (-2,\y) circle (3pt);
        }

        \foreach \y  in {-2,-1,1,2}
        {
        \filldraw [fill=white, draw=black,thick] (2,\y) circle (3pt);
        }

    \draw (-0.8,0.2) node {{\bf 1}}; \draw (0.2,0.2) node {{\bf 2}}; \draw (1.2,0.2) node {{\bf 3}};
    \draw (2.2,0.2) node {{\bf 4}};
    \draw (-0.8,1.2) node {{\bf 15}}; \draw (0.2,1.2) node {{\bf 5}}; \draw (1.2,1.2) node {{\bf 6}};
    \draw (-0.8,-0.8) node {{\bf 14}}; \draw (0.2,-0.8) node {{\bf 7}}; \draw (1.2,-0.8) node {{\bf 8}};
    \draw (-0.8,2.2) node {{\bf 16}}; \draw (0.2,2.2) node {{\bf 10}}; \draw (1.2,2.2) node {{\bf 12}};
    \draw (-0.8,-1.8) node {{\bf 13}}; \draw (0.2,-1.8) node {{\bf 9}}; \draw (1.2,-1.8) node {{\bf 11}};

        \end{tikzpicture}
    \end{center}
    \caption{A dominating sequence of $C_5 \cp C_5$}
    \label{fig:CtimesC}
\end{figure}

Now we prove that the left-hand side in \textbf{(a)}, \textbf{(b)}, \textbf{(c)}, and \textbf{(d)} is at most the right-hand side. We will use Lemma~\ref{boundary}.

To prove \textbf{(a)} let $S=((a_1,b_1),(a_2,b_2),\dots,(a_m,b_m))$ be a Grundy dominating sequence of $P_k \cp P_l$. Let $i$ be the minimum index such that either there is a column containing $l$ vertices of $\widehat{S_i}$ or a row containing $k$ vertices of $\widehat{S_i}$. If there is no such $i$, then $m \le k(l-1)$ and we are done.
Note that the two cases cannot happen simultaneously at any step $i$ since otherwise all the neighbors of $(a_i,b_i)$ and $(a_i,b_i)$ itself would already be dominated by $\widehat{S_{i-1}}$ and so $S$ would not be a legal sequence. So without loss of generality we can assume that there is horizontal segment of length $k$. Since in each column of $P_k \cp P_l$ there is at least one and at most $l-1$ elements of the set $\widehat{S_i}$, we have that $|\partial \widehat{S_i}| \ge k$, which---by Lemma \ref{boundary}---proves~\textbf{(a)}.

To prove \textbf{(b)} let $S=((a_1,b_1),(a_2,b_2),\dots,(a_m,b_m))$ be a Grundy dominating sequence of $P_k \cp C_l$.  Let $i$ be the minimum index such that either there is a column containing $l-1$ vertices of $\widehat{S_i}$ or a row containing $k$ vertices of $\widehat{S_i}$. There exists such an $i$ as otherwise $m\le\max\{l(k-1),k(l-2)\} $ and we are done. We observe again that both cases cannot happen at the same time since otherwise all the neighbors of $(a_i,b_i)$ and $(a_i,b_i)$ itself would already be dominated. Suppose first that there is a column containing $l-1$ vertices of $\widehat{S_i}$. Then in each row there is an element of the boundary of $\widehat{S_i}$. By Lemma~\ref{boundary}, this means that the length of the Grundy dominating sequence is at most $l(k-1)$. Now suppose that there is column containing $k$ vertices of $\widehat{S_i}$. Then---since there is an element in each column and there is no column containing $l-1$ vertices of $\widehat{S_i}$---we have that the cardinality of the boundary of $\widehat{S_i}$ is at least $2k$. By Lemma~\ref{boundary}, the length of the Grundy dominating sequence is at most $k(l-2)$ and we are done with \textbf{(b)}.

Now we prove \textbf{(c)}.

\medskip\noindent
\textsc{Case I.} $k+1 \le l$. \\
Let $S=(a_1,b_1),(a_2,b_2),\dots,(a_m,b_m)$ be a Grundy dominating sequence of $C_k \cp C_l$.  Let $i_1$ be the minimum index such that either there is a column containing $l-1$ vertices of $\widehat{S_i}$ or there is a row containing $k-1$ vertices of $\widehat{S_i}$. Note that it cannot happen that both of these hold as then $(a_{i_1},b_{i_1})$ would not dominate any new vertex.

\medskip\noindent
\textsc{Subcase IA.} There exists a column containing $l-1$ vertices of $\widehat{S_{i_1}}$.\\
Then $|\partial \widehat{S_{i_1}} | \ge 2(l-1)\ge 2k$ and we are done by Lemma~\ref{boundary}.

\medskip\noindent
\textsc{Subcase IB1.} There exists a row $P_k^v$ containing $k-1$ vertices of $\widehat{S_{i_1}}$, and before the appearance of another such row and before the first row containing $k$ vertices appears, there exists a column $^uP_l$ containing $l-1$ vertices of $\widehat{S_{i_2}}$ for some $i_2>i_1$. \\
Then the boundary of $\widehat{S_{i_2}}$ contains $^uP_l\setminus \widehat{S_{i_2}}$, $P_k^v\setminus \widehat{S_{i_2}}$ and 2 vertices from every other row. Therefore  $|\partial \widehat{S_{i_2}}| \ge 1+1+2(l-2) \ge 2k$ holds and we are done by Lemma~\ref{boundary}.

\medskip\noindent
\textsc{Subcase IB2.} There exists a row $P_k^v$ containing $k-1$ vertices of $\widehat{S_{i_1}}$, and before the appearance of another such row and before a column $^uP_l$ containing $l-1$ vertices appears, the vertex in $P_k^v\setminus \widehat{S_{i_1}}$ becomes the $i_2$nd vertex of $S$ for some $i_2>i_1$. \\
Then the boundary of $\widehat{S_{i_2}}$ contains two vertices from every column and therefore $|\partial \widehat{S_{i_2}}|\ge 2k$ and we are done by Lemma~\ref{boundary}.

\medskip\noindent
\textsc{Subcase IB3.} There exists a row $P_k^v$ containing $k-1$ vertices of $\widehat{S_{i_1}}$, and before the appearance of a column $^uP_l$ containing $l-1$ vertices
and before the vertex in $P_k^v\setminus \widehat{S_{i_1}}$ appears in $S$, there exists another row $P_k^w$ containing $k-1$ vertices from $\widehat{S_{i_2}}$ for some $i_2>i_1$.\\
Then either every column contains a vertex from $\widehat{S_{i_2}}$ and then $|\partial \widehat{S_{i_2}}|\ge 2k$ or the vertices of $P_k^v\setminus \widehat{S_{i_2}}$ and $P_k^w\setminus \widehat{S_{i_2}}$ are in the same column. Then these two vertices are in the boundary of $\widehat{S_{i_2}}$ and two vertices from every other column belongs to $\partial \widehat{S_{i_2}}$ and thus $|\partial \widehat{S_{i_2}}|\ge 2+2(k-1)=2k$. We are done by Lemma~\ref{boundary}.

\medskip\noindent
\textsc{Case II.} $k=l$ even.\\
Let $k=2t$, $V(C_k \cp C_k)=\{(x,y) \in \mathbb{Z}^2 : -t+1 \le x,y \le t\}$ and $E(C_k \cp C_k)=\{((x,y),(x',y')): |x-x'|+|y-y'|=1, \ \textrm{or} \ x=x' \ \textrm{and} \ |y-y'|=k, \ \textrm{or} \ y=y' \ \textrm{and} \ |x-x'|=k \}$. Let $S=(a_1,b_1),(a_2,b_2),\dots,(a_m,b_m)$ be a Grundy dominating sequence of $C_k \cp C_k$ and let $i$ be the minimum index such that either a row or a column contains $k-1$ vertices from $\widehat{S_i}$. As in the previous case, there cannot be both a row and a column that contain $k-1$ vertices of $\widehat{S_i}$ as $N[(a_i,b_i)]$ would already be dominated by $\widehat{S_{i-1}}$. We can assume without loss of generality that $\widehat{S_i}$ contains $\{(0,y): |y|\le t-1\}$. As then every row contains at most $k-2$ vertices of $\widehat{S_i}$, we obtain that $\partial \widehat{S_i}$ contains 2 vertices in each row except the row $R_t=\{(x,t): -t+1 \le x \le t\}$. The row $\{(x,y):-t+1\le x\le t\}$ will be denoted by $R_y$. As $(0,t)\in \partial \widehat{S_i}$, we have $|\partial \widehat{S_i}|\ge 2k-1$. If $|\partial \widehat{S_i}|\ge 2k$ holds, then we are done by Lemma~\ref{boundary}. Otherwise, we must have $|\partial \widehat{S_i}|= 2k-1$ and $\partial \widehat{S_i} \cap R_t=\{(0,t)\}$. This implies $R_{-t+1}\cap \widehat{S_i}=\{(0,-t+1)\}$ and $R_{t-1}\cap \widehat{S_i}=\{(0,t-1)\}$. Note that if for some row $|R_u\cap \widehat{S_i}|<k-2$, then $R_u\cap \widehat{S_i}$ must be an interval as otherwise $|R_u\cap \partial \widehat{S_i}|\ge 3$ and thus $|\partial \widehat{S_i}|\ge 2k$ would hold. Also, if for some row $|R_u\cap \widehat{S_i}|<k-2$ and $R_u\cap \widehat{S_i}=\{(x,u):\alpha\le x\le \beta\}$, then $R_{u-1}\cap \widehat{S_i}\subseteq \{(x,u-1):\alpha-1\le x\le \beta+1\}$ and $R_{u+1}\cap \widehat{S_i}\subseteq \{(x,u+1):\alpha-1\le x\le \beta+1\}$ hold as otherwise $|R_u\cap \partial \widehat{S_i}|\ge 3$ and thus $|\partial \widehat{S_i}|\ge 2k$ would hold.

From the above it follows that $\widehat{S_i}\subseteq T:=\{(x,y) : |x|+|y| \le t-1\}$ holds, see Figure~\ref{fig:C6timesC6}. However the number of vertices in $T$ is less than $k(k-2)$, so there must exist a minimal $j$ such that $\widehat{S_j}$ contains a vertex outside of $T$. An esay case analysis shows that $|\partial \widehat{S_j}| \ge 2k$ holds and therefore we are done by Lemma~\ref{boundary}.

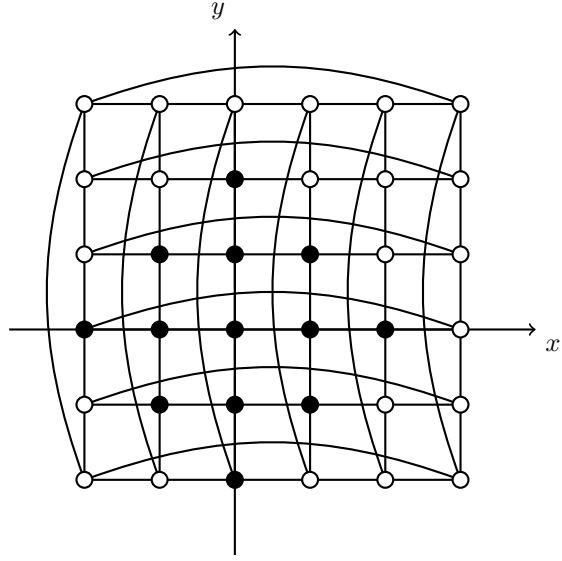
\begin{figure}[ht!]
    \begin{center}
        \begin{tikzpicture}[scale=1.0,style=thick,x=1cm,y=1cm]
        \def\vr{2.5pt} 
        \draw[thick, ->] (-3,0) -- (4,0) node[anchor = north west] {$x$};
    \draw[thick, ->] (0,-3) -- (0,4) node[anchor = south east] {$y$};

    \foreach \x  in {-2,-1,0,1,2,3}
    {
    \draw (\x,-2) -- (\x,3);
    \draw (\x,-2) to [bend left=20] (\x,3);
    }
    \foreach \y  in {-2,-1,0,1,2,3}
    {
    \draw (-2,\y) -- (3,\y);
    \draw (-2,\y) to [bend left=20] (3,\y);
    }

        \foreach \x  in {-2,-1,0,1,2}
        {
        \filldraw [fill=black, draw=black,thick] (\x,0) circle (3pt);
        }
        \foreach \x  in {-1,0,1}
        \foreach \y in {-1,1}
        {
        \filldraw [fill=black, draw=black,thick] (\x,\y) circle (3pt);
        }

        \foreach \y in {-2,2}
        {
        \filldraw [fill=black, draw=black,thick] (0,\y) circle (3pt);
        }

      \foreach \x  in {-2,-1,0,1,2,3}
        {
        \filldraw [fill=white, draw=black,thick] (\x,3) circle (3pt);
        }
        \foreach \x  in {-2,-1,1,2,3}
        \foreach \y in {-2,2}
        {
        \filldraw [fill=white, draw=black,thick] (\x,\y) circle (3pt);
        }

        \foreach \x  in {-2,2,3}
        \foreach \y in {-1,1}
        {
        \filldraw [fill=white, draw=black,thick] (\x,\y) circle (3pt);
        }
        \filldraw [fill=white, draw=black,thick] (3,0) circle (3pt);

        \end{tikzpicture}
    \end{center}
    \caption{A set $T$ in $C_6 \cp C_6$}
    \label{fig:C6timesC6}
\end{figure}

Finally, let us prove the upper bound in \textbf{(d)}. Let $S=(a_1,b_1),(a_2,b_2),\dots,(a_m,b_m)$ be a Grundy dominating sequence of $C_k \cp C_k$ for odd $k$. Let $i$ be the minimum index such that there is either a row or a column containing $k-1$ vertices from $\widehat{S_i}$. We observe again that both cannot happen at step $i$ since otherwise all the neighbors of $(a_i,b_i)$ and $(a_i,b_i)$ itself would already be dominated. So we can assume without loss of generality that there exists a column containing $k-1$ vertices from $\widehat{S_i}$. Then in each row we have at least two elements in the boundary of $\widehat{S_i}$, except one, where we have at least one element. By this we have $|\partial \widehat{S_i}| \ge 2k-1$, and the upper bound of \textbf{(d)} follows by Lemma~\ref{boundary}.
\qed

In the remainder of  this section we show how Lemma~\ref{boundary}
and isoperimetric inequalities can be applied to prove results on
the Grundy domination number of products of several paths and
products of several even cycles. For a graph $G$, a vertex $v \in
V(G)$ and natural number $r$ let us denote by $B(G,v,r)$ those
points in $G$, whose distance from $v$ is at most $r$. We will need
the following two isoperimetric theorems that are consequences of
the cited results for graphs of special vertex cardinalities.

\begin{theorem}{\rm (Bollob\'as, Leader, \cite[Theorem 8]{bl1991})}
\label{iso1}
Suppose that $H \subset C_{2k_1} \cp \cdots \cp C_{2k_n}$ and $|V(H)|=|B(C_{2k_1} \cp \cdots \cp C_{2k_n},v,r)|$ for some $v \in V(C_{2k_1}  \cp \cdots \cp C_{2k_n})$ and a natural number $r$. Then $$|\partial H| \ge |\partial{B(C_{2k_1} \cp \cdots \cp C_{2k_n},v,r)}|.$$
\end{theorem}

\begin{theorem}{\rm (Riordan, \cite[Theorem 1.1]{riordan-1998})}
\label{iso2}
Suppose that $H \subset P_{k_1} \cp \cdots \cp P_{k_n}$ and $|V(H)|= |B(P_{k_1} \cp \cdots \cp P_{k_n},v,r)|$ for some vertex $v \in V(P_{k_1} \cp \cdots \cp P_{k_n})$ of minimum degree and some natural number $r$. Then $$|\partial H| \ge |\partial{B(P_{k_1} \cp \cdots \cp P_{k_n},v,r)}|.$$
\end{theorem}

\begin{proposition}
If $k_1 \le \dots \le k_n$ and $\sum_{i=1}^{n-1}k_i +2 \le k_n$, then we have $$\gamma_{gr}(C_{2k_1} \cp \cdots \cp C_{2k_n})=2^nk_1 \cdots k_{n-1}(k_n-1).$$  If $k_1 \le \dots \le k_n$ and $\sum_{i=1}^{n-1}k_i + 1 \le k_n$, then we have $$\gamma_{gr}(P_{k_1} \cp \cdots \cp P_{k_n})= k_1 \cdots k_{n-1}(k_n-1).$$
\end{proposition}

\proof In both statements, the lower bound on the Grundy domination number follows from Proposition \ref{prp:cart}.

The upper bound in the first statement follows from Lemma~\ref{boundary} and Theorem \ref{iso1}, while the upper bound of the second statement follows from Lemma~\ref{boundary} and Theorem \ref{iso2}.
\qed

\section{Lexicographic product}
\label{sec:lex}

In this section we give an expression for the Grundy domination
number of the lexicographic product of graphs in terms of
corresponding invariants of the factors. From this result we in
particular did use explicit formulas for the Grundy domination
number of grids, cylinders and tori (with respect to the
lexicographic product).

\begin{proposition}
\label{pr:lex}
Let $G$ and $H$ be graphs. Then $$\gamma_{gr}(G\circ H)\ge \max\{{\alpha}(G)\gamma_{gr}(H),\gamma_{gr}(G)\}.$$
\end{proposition}
\proof
Let $\ell={\alpha}(G),n= \gamma_{gr}(H)$, let $\{d_1,\ldots,d_{\ell}\}$ be a maximum independent set in $G$,
and $(d'_1,\ldots,d'_n)$ be a Grundy dominating sequence in $H$.
Observe that $$((d_1,d'_1),\ldots,(d_1,d'_n),(d_2,d'_1),\ldots,(d_2,d'_n),\ldots,(d_{\ell},d'_1),\ldots,(d_{\ell},d'_n))$$
is a dominating sequence in $G\circ H$. Hence $\gamma_{gr}(G\circ H)\ge \ell n=\alpha(G)\gamma_{gr}(H)$.

The bound $\gamma_{gr}(G\circ H)\ge \gamma_{gr}(G)$ is trivial.
\qed

The lower bound from Proposition~\ref{pr:lex} is sharp. For instance, if $H$ is a complete graph
then clearly $\gamma_{gr}(G\circ H)=\gamma_{gr}(G)$. It is easy to see that  $\gamma_{gr}(G)\gamma_{gr}(H)$
is an upper bound for $\gamma_{gr}(G\circ H)$. Thus for any graph $G$, in which $\gamma_{gr}(G)=\alpha(G)$,
we have $\gamma_{gr}(G\circ H)= {\alpha}(G)\gamma_{gr}(H)$. On the other hand, there are graphs in which
this bound is not sharp. For instance, let $T$ be the tree from Fig.~\ref{fig:tree}, and let $H$ be any
graph, which is not complete, thus $\gamma_{gr}(H)\ge 2$. Note that the filled vertices in the figure present
a maximum independent set of $T$, and so $\alpha(T)=5$. Hence the bound in Proposition~\ref{pr:lex} is
$\alpha(T)\gamma_{gr}(H)=5\gamma_{gr}(H)$, but $\gamma_{gr}(T\circ H)\ge 5\gamma_{gr}(H)+1$. Indeed,
let $S$ be the sequence that starts in the layer $^u\!H$, by legally picking $\gamma_{gr}(H)$ vertices,
then choosing a vertex from $^v\!H$, and then choosing $\gamma_{gr}(H)$ vertices in each of the layers
that correspond to the remaining four leaves of $T$. This yields a legal sequence of the desired length.

\begin{figure}[ht!]
\centering
\begin{tikzpicture}[scale=.9,style=thick]
\def\vr{2.5pt}


\coordinate (p1) at (-1,0);
\coordinate (p2) at (0,0);
\coordinate (p3) at (1,1);
\coordinate (p4) at (2,1.4);
\coordinate (p5) at (2,0.6);
\coordinate (p6) at (1,-1);
\coordinate (p7) at (2,-1.4);
\coordinate (p8) at (2,-0.6);

\draw (p1) -- (p2) -- (p3) -- (p4);
\draw (p3) -- (p5);
\draw (p2) -- (p6) -- (p7);
\draw (p6) -- (p8);


\draw(p1)[fill=black] circle(\vr);
\draw(p2)[fill=white] circle(\vr);
\draw(p3)[fill=white] circle(\vr);
\draw(p4)[fill=black] circle(\vr);
\draw(p5)[fill=black] circle(\vr);
\draw(p6)[fill=white] circle(\vr);
\draw(p7)[fill=black] circle(\vr);
\draw(p8)[fill=black] circle(\vr);

\node[below] at (p2){$v$};
\node[below] at (p1){$u$};

\end{tikzpicture}
\vskip -0.3 cm \caption{The tree $T$.} \label{fig:tree}
\end{figure}
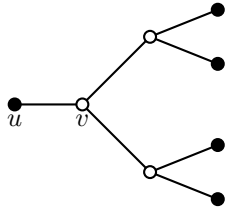

Next, we present an exact formula for $\gamma_{gr}(G\circ H)$.
Given a dominating sequence $D=(d_1,\ldots,d_k)$ in a graph $G$,
let $a(D)$ denote the cardinality of the set of vertices $d_i$ from $D$,
which are not adjacent to any vertex from $\{d_1,\ldots,d_{i-1}\}$.

\begin{theorem}
\label{thm:lex}
Let $G$ and $H$ be graphs. Then $$\gamma_{gr}(G\circ H)= \max\{a(D)(\gamma_{gr}(H)-1)+|\widehat{D}|\,;\,D\textit{ is a dominating sequence of $G$}\}.$$
\end{theorem}
\proof
Let $D=(d_1,\ldots,d_m)$ be any dominating sequence of $G$,
and let $(d'_1,\ldots,d'_k)$ be a Grundy dominating sequence of $H$.
Then one can find a dominating sequence $S$ in $G\circ H$
of length $a(D)(\gamma_{gr}(H)-1)+|\widehat{D}|$ as follows.
Let $S$ be the sequence that corresponds to $D$,
only those vertices $d_i\in \widehat{D}$ which are not adjacent
to any vertex from $\{d_1,\ldots,d_{i-1}\}$ are repeated
$\gamma_{gr}(H)$ times in a row, so that the corresponding subsequence
is of the form  $((d_i,d'_1),\ldots,(d_i,d'_k))$. On the other hand,
the vertices $d_i\in \widehat{D}$ which are adjacent to some vertex from
$\{d_1,\ldots,d_{i-1}\}$ are projected only once, from any vertex
$(d_i,h)\in\ ^{d_i}\!H.$ It is easy to see that in either case the vertices
in $S$ are legally chosen. In the first case this is true because no vertex
of $^{d_i}\!H$ is dominated at the point when $(d_i,d'_1)$ is chosen,
thus $((d_i,d'_1),\ldots,(d_i,d'_k))$ is a legal subsequence. In the second
case this is true because $D$ is a legal sequence in $G$, and so when $d_i$
is chosen, which is at that point already dominated, this implies that
there exists another vertex $t\in V(G)$ that $d_i$ footprints. Hence,
when $(d_i,h)$ is chosen in $S$, it footprints vertices from $^{t}\!H$.
Note that the length of $S$ is $a(D)(\gamma_{gr}(H)-1)+|\widehat{D}|$,
and that $D$ is an arbitrary legal sequence of $G$.
This implies that
$\gamma_{gr}(G\circ H)\ge \max\{a(D)(\gamma_{gr}(H)-1)+|\widehat{D}|$.

For the converse, let $S$ be an arbitrary dominating sequence in $G\circ H$. Let $s_i=(x,y)$ be a vertex from $S$,
where $x\in V(G),y\in V(H)$. Note that when $s_i$ is added to $S$, all vertices from the layers $^{h}\!H$,
where $h\in N(x)$, are dominated. In particular, at most one vertex from each of the layers $^{h}\!H$
can be in $S$ after $s_i$.

Now, consider the sequence $D$ of vertices from $G$, defined as follows:
for each $s_i=(x,y)\in S$ add $x$ to $D$ if $s_i$ is the first vertex from $S$ which is in $^x\!H$.
First note that $D$ is a legal sequence of $G$. Indeed, when $x$ is added to $D$, $(x,y)$ footprints
some vertex, either from $^{x}\!H$ or from $^{z}\!H$, where $z\in N(x)$.
If $(x,y)$ footprints a vertex $(x,h)$, this implies that no vertex $s_j\in S$, where $j<i$, is in $N[x]\times V(H)$.
Thus in this case $x$ footprints itself with respect to $D$. In the case when $(x,y)$ footprints a vertex from a layer $^{g}\!H$,
where $g\in N(x)$, $x$ footprints $g$ with respect to $D$. In either case $x$ is a legal choice, hence $D$ is a legal sequence in $G$.
It is clear that $\widehat{D}$ is a dominating set, since $\widehat{D}=p_G(\widehat{S})$ and $\widehat{S}$ is a dominating set of
$G\circ H$.

Let $A(D)$ be the set of all vertices $d_k$ from $D$, such that $d_k$ is not adjacent to any vertex from $\{d_1,\ldots,d_{k-1}\}$.
(Note that $|A(D)|=a(D)$ by definition.) By the way $D$ is constructed, if $d_k\not\in A(D)$, then at the point when $s_i=(d_k,y)$ is added to
$S$, all vertices of the layer $^{d_k}\!H$ are already dominated. Hence, at most one vertex from $^{d_k}\!H$ can lie in $S$.
On the other hand, clearly at most $\gamma_{gr}(H)$ vertices from layers $^{d_k}\!H$, where $d_k\in A(D)$, can lie in $S$.
We infer that $|\widehat{S}|\le (|\widehat{D}|-a(D))+a(D)\gamma_{gr}(H)$, from which the desired inequality follows.
\qed

\begin{corollary}\label{cor:PnH} Let $H$ be an arbitrary graph that is not a complete graph. Then
\begin{displaymath}
\ggr(P_k \circ H)=
\left\{ \begin{array}{l l}
\frac{k}{2}\cdot \ggr(H)+1, & \textrm{$k$ is even, $k\neq 2$}\\
\left\lceil \frac{k}{2} \right\rceil \cdot \ggr(H), & \textrm{$k$ is odd.}\\
\end{array}
\right.
\end{displaymath}
\end{corollary}

\proof
Let $P_k=v_1,v_2,\ldots ,v_k.$
Using Theorem~\ref{thm:lex}, we would like to find a dominating sequence $D=(d_1,\ldots , d_t)$ of $P_k$ with the largest value $a(D)\ggr(H)+(|\widehat{D}|-a(D))$. From the proof of Theorem~\ref{thm:lex} it follows that the vertices $d_i \in \widehat{D}$ that are not adjacent to vertices $\widehat{D_{i-1}}$ contribute $\ggr(H) \geq 2$ to the Grundy dominating sequence of $P_k \circ H$, but the other vertices of $\widehat{D}$ contribute just 1 to the Grundy dominating sequence of $P_k \circ H$. Therefore $D$ is optimal if and only if $a(D)$ is as big as possible. Thus $a(D)=\left\lceil  \frac{k}{2} \right\rceil$. In the case when $k$ is odd, $D=(v_1,v_3,\ldots , v_k)$ and this is the only case with $a(D)=\frac{k+1}{2}.$ Thus $\ggr(P_k\circ H)=\left\lceil \frac{k}{2} \right\rceil \cdot \ggr(H).$ If $k$ is even, then $a(D)=\frac{k}{2}$ implies $|\widehat{D}| \leq \frac{k}{2}+1$. The set $D$ can be chosen in such a way that the equality holds, for example $D=(v_1,v_2,v_4,v_6,\ldots , v_k)$ contains $\frac{k}{2}+1$ vertices. Thus $\ggr(P_k\circ H)=\frac{k}{2} \cdot \ggr(H) +1.$
\qed

\begin{corollary} Let $k,l > 2$. Then
\begin{displaymath}
\ggr(P_k \circ P_l)=
\left\{ \begin{array}{l l}
\frac{k}{2}\cdot (l-1)+1, & \textrm{$k$ is even}\\
\left\lceil \frac{k}{2} \right\rceil \cdot (l-1), & \textrm{$k$ is odd.}\\
\end{array}
\right.
\end{displaymath}
\end{corollary}

\begin{corollary} Let $k,l > 2$. Then
\begin{displaymath}
\ggr(P_k \circ C_l)=
\left\{ \begin{array}{l l}
\frac{k}{2}\cdot (l-2)+1, & \textrm{$k$ is even}\\
\left\lceil \frac{k}{2} \right\rceil \cdot (l-2), & \textrm{$k$ is odd.}\\
\end{array}
\right.
\end{displaymath}
\end{corollary}

Using Theorem~\ref{thm:lex} and the same ideas as in the proof of Corollary~\ref{cor:PnH}, we obtain the following result.

\begin{corollary}
Let $H$ be an arbitrary graph that is not a complete graph, and let $k>3$. Then
\begin{displaymath}
\ggr(C_k \circ H)=
\left\{ \begin{array}{l l}
\frac{k}{2}\cdot \ggr(H), & \textrm{$k$ is even}\\
\left\lfloor  \frac{k}{2} \right\rfloor \cdot \ggr(H) +1, & \textrm{$k$ is odd.}\\
\end{array}
\right.
\end{displaymath}
\end{corollary}

\begin{corollary} Let $k,l > 3$. Then
\begin{displaymath}
\ggr(C_k \circ C_l)=
\left\{ \begin{array}{l l}
\frac{k}{2}\cdot (l-2), & \textrm{$k$ is even}\\
\left\lfloor  \frac{k}{2} \right\rfloor \cdot (l-2) +1, & \textrm{$k$ is odd.}\\
\end{array}
\right.
\end{displaymath}
\end{corollary}

\section{Direct product}
\label{sec:direct}

Quite often the direct product is the most difficult one among the
standard products when investigating its invariants and similar
problems.  This phenomenon is also true in the case of the Grundy
domination number. In this section we give a general lower bound on
the Grundy domination number and specialize it to products of paths
and cycles. In the case of the product of two paths where a shortest
path is of even order we also give an exact value.
  We begin
with the following general lower bound where $a(D)$ is the same
function as in Section~\ref{sec:lex}.

\begin{proposition}\label{direct}
Let $G$ and $H$ be graphs. Then
\begin{eqnarray*}
     \begin{aligned}
     &\gamma_{gr}(G\times H) \geq \max \big\{ \\
        &\max\{a(D)|V(H)|+\ggr^t(H)(|\widehat{D}|-a(D))\,;\,D\textrm{ is a dominating sequence of $G$}\}\\
        &\max\{a(D)|V(G)|+\ggr^t(G)(|\widehat{D}|-a(D))\,;\,D\textrm{ is a dominating sequence of $H$}\} \big\}\\
     \end{aligned}
\end{eqnarray*}
\end{proposition}

\proof We first note that it suffices to present a construction that yields a dominating sequence of length $$\max\{a(D)|V(H)|+\ggr^t(H)(|\widehat{D}|-a(D))\,;\,D\textrm{ is a dominating sequence of $G$}\}\,.$$
Let $D$ be a dominating sequence of $G$. Now we construct a sequence $S$ in $G \times H$ that corresponds to $D$ in the following way. For a vertex $d_i\in \widehat{D}$, which is not adjacent
to any vertex from $\{d_1,\ldots,d_{i-1}\}$, all vertices of $^{d_i}\!H$ are added to $S$ (this is legal because $^{d_i}\!H$ is an independent set, and $(d_i,h)$ is not adjacent to any $(d_j,h')$ for $j<i$). On the other hand, if $d_i$ is adjacent to some $d_j$, $j<i$, then $d_i$ footprints a vertex $g$ with respect to $D$ in $G$. Given a Grundy total dominating sequence $T=(t_1,\ldots , t_r)$ in $H$, where $r=\ggr^t(H),$ we add to $S$ the sequence $(d_i,t_1),\ldots , (d_i,t_r)$ (indeed this is legal because $(d_i,t_j)$ footprints $(g,t_j')$, where $t_j'$ is a neighbor of $t_j$ footprinted by $t_j$ with respect to $T$). The length of $S$ is $a(D)|V(H)|+\ggr^t(H)(|\widehat{D}|-a(D))$, as desired. \qed

In the next few results we will use the following notations $V(P_k)=[k]$, $E(P_k)=\{\{i,i+1\}: 1\le i\le k-1\}$ and $V(C_k)=[k]$, $E(C_k)=\{\{i,i+1\}: 1\le i\le k-1\} \cup \{1,k\}.$

Proposition~\ref{direct} yields the following lower bounds.
\begin{corollary}
If $k \geq 2$ and $l\geq 4$, then
\begin{displaymath}
\ggr(P_k \times C_l) \geq
\left\{ \begin{array}{l l}
\max\{kl-2k-l+6,kl-2k\}, & \textrm{$k,l$ are even}\\
kl-k-l+3, & \textrm{$k,l$ are odd.}\\
\max\{kl-2k,kl-k-l+3\}, & \textrm{$k$ is even, $l$ is odd}\\
kl-2k-l+6, & \textrm{$k$ is odd, $l$ is even}\\
\end{array}
\right.
\end{displaymath}
\end{corollary}
\proof
The proof of all four cases is similar, the only difference is in total Grundy domination number of paths and cycles with respect to the parity of the length. That is
\begin{displaymath}
\ggr^t(P_k) =
\left\{ \begin{array}{l l}
k, & \textrm{$k$ is even}\\
k-1, & \textrm{$k$ is odd}\\
\end{array}
\right.
\end{displaymath}
and
\begin{displaymath}
\ggr^t(C_l) =
\left\{ \begin{array}{l l}
l-2, & \textrm{$l$ is even}\\
l-1, & \textrm{$l$ is odd.}\\
\end{array}
\right.
\end{displaymath}
Let $D=(1,\ldots , k-2,k)$ be a dominating sequence of $P_k$ and $D'=(1,\ldots , l-2)$ a dominating sequence of $C_l$. Then $a(D)=2$ and $a(D')=1.$ First let $k$ and $l$ be even. Since $\ggr^t(C_l)=l-2$, Proposition~\ref{direct} implies that $\ggr(P_k \times C_l) \geq  a(D)|V(C_l)|+\ggr^t(C_l)(|\widehat{D}|-a(D)) = 2l+ (k-3)(l-2)=kl-2k-l+6.$ Since $\ggr^t(P_k)=k$ it follows from Proposition~\ref{direct} that $\ggr(P_k\times C_l) \geq a(D')|V(P_k)|+\ggr^t(P_k)(|\widehat{D'}|-a(D'))= k+(l-3)k=kl-2k.$

If $k$ is even and $l$ is odd, then $\ggr(P_k \times C_l) \geq  a(D)|V(C_l)|+\ggr^t(C_l)(|\widehat{D}|-a(D)) = 2l+(k-3)(l-1)= kl-k-l+3$ and $\ggr(P_k\times C_l) \geq a(D')|V(P_k)|+\ggr^t(P_k)(|\widehat{D'}|-a(D'))= k+ k(l-3)=kl-2k.$

If $k$ is odd and $l$ is even, then $\ggr(P_k \times C_l) \geq  a(D)|V(C_l)|+\ggr^t(C_l)(|\widehat{D}|-a(D)) = 2l+(k-3)(l-2)= kl-2k-l+6$.

Finally let $k$ and $l$ be odd. Then $\ggr(P_k \times C_l) \geq  a(D)|V(C_l)|+\ggr^t(C_l)(|\widehat{D}|-a(D)) = 2l+(k-3)(l-1)= kl-k-l+3$. \qed

The proof of the next result is omitted, as it follows from the Grundy total domination number of a cycle and a dominating sequence of a cycle $C_n$ of length $n-2.$

\begin{corollary}
If $l \geq k \geq 4$, then
\begin{displaymath}
\ggr(C_k \times C_l) \geq
\left\{ \begin{array}{l l}
kl-2k-2l+6, & \textrm{$k,l$ are even}\\
kl-2k-l+3, & \textrm{$k$ is odd.}\\
kl-k-2l+3, & \textrm{$k$ is even, $l$ is odd.}\\
\end{array}
\right.
\end{displaymath}
\end{corollary}

\begin{corollary}\label{direct:path}
If $2\le k\le l$, then
\begin{displaymath}
\ggr(P_k \times P_l) \geq
\left\{ \begin{array}{l l}
kl-k, & \textrm{$k$ is even}\\
kl-k-l+3, & \textrm{$k,l$ are odd.}\\
\max\{kl-l, kl-k-l+3\}, & \textrm{$k$ is odd, $l$ is even.}\\
\end{array}
\right.
\end{displaymath}
\end{corollary}
\proof
Let $D=(1,\ldots , k-2,k)$ be a dominating sequence of $P_k$ and $D'=(1,\ldots , l-2,l)$ a dominating sequence of $P_l$. Then $a(D)=a(D')=2$. If $k$ is even, then it follows from Proposition~\ref{direct} that $\ggr(P_k\times P_l) \geq a(D')|V(P_k)|+\ggr^t(P_k)(|\widehat{D'}|-a(D'))= 2k+(l-3)k=kl-k.$ If both $k$ and $l$ are odd, then
$\ggr(P_k \times P_l) \geq  2l+(k-3)(l-1)= kl-k-l+3.$ If $k$ is odd and $l$ is even, then $\ggr(P_k \times P_l) \geq \max\{ 2l+(k-3)l, 2k+(l-3)(k-1)\}= \max\{kl-l,kl-k-l+3\}.$

\qed

\begin{proposition}\label{direct:upper}
If $2\le k\le l$, then $\ggr(P_k\times P_l)\le kl-k$.
\end{proposition}

\proof
The proof is very similar to that of the upper bound in Theorem \ref{cartesiangrid} \textbf{(a)}. The graph $P_k\times P_l$ has two components $C_{odd}=\{(a,b): a+b ~\textnormal{is odd}\}$ and $C_{even}=\{(a,b): a+b ~\textnormal{is even}\}$. We will prove that the length of a dominating sequence is $\lfloor k/2\rfloor$ less than the number of vertices in one of them and $\lceil k/2\rceil$ less than the number of vertices in the other. Let us consider the component $C_{odd}$. Let $L_{odd,+}$ denote the set of lines $\{(x,y): x+y=c\}$ where $c$ is an odd integer with $\lfloor k/2+1\rfloor \le c\le \lfloor 3k/2+1\rfloor$, and let $L_{odd,-}$ denote the set of lines $\{(x,y): x-y=c\}$ where $c$ is an odd integer with $1-\lfloor k/2\rfloor \le c\le \lceil k/2\rceil$. Finally, let $S=((a_1,b_1),(a_2,b_2),\dots,(a_m,b_m))$ be a dominating sequence of $C_{odd}$ and let $i$ be the smallest index for which either every line in $L_{odd,-}$ contains a vertex from $\widehat{S_i}$ or every line in $L_{odd,+}$ contains a vertex from $\widehat{S_i}$ (or both). Such an index exists as if not then the vertex $\ell_-\cap\ell_+$ with $\ell_-\in L_{odd,-}, \ell_+\in L_{odd,+}$ is not dominated by $\widehat{S}$ if none of $\ell_-$ and $\ell_+$ contain a vertex from $\widehat{S}$. Observe that $|\partial \widehat{S_i}|\ge \min\{|L_{odd,-}|,|L_{odd,+}|\}$ holds as if every line in $L_{odd,-}$ contains a vertex from $\widehat{S_i}$, then every line in $L_{odd,+}$ contains a vertex from $\partial \widehat{S_i}$ and vice versa.
\qed

Corollary~\ref{direct:path} and Proposition~\ref{direct:upper} give the following exact result.

\begin{corollary}
Let $2 \leq k \leq l$ and let $k$ be even. Then $\ggr(P_k \times P_l)=kl-k.$
\end{corollary}

\section{Strong product}
\label{sec:strong}

In this section we first observe that $\gamma_{gr}(G\boxtimes H)\ge
\gamma_{gr}(G)\gamma_{gr}(H)$ holds for any graphs $G$ and $H$, and
conjecture that it always holds with equality. Among other results
proved here we confirm the conjecture for strong products of
caterpillars with arbitrary graphs.

\begin{proposition}\label{strong}
For any  graphs $G$ and $H$, $$\gamma_{gr}(G\boxtimes H)\ge
\gamma_{gr}(G)\gamma_{gr}(H).$$
\end{proposition}
\proof Let $p=\gamma_{gr}(G),q= \gamma_{gr}(H)$, let
$(d_1,\ldots,d_p)$ be a Grundy dominating sequence in $G$, and
$(d'_1,\ldots,d'_q)$ be a Grundy dominating sequence in $H$. Let
$d_i$ footprint $u_i$, $i \in [p]$ and let $d'_j$ footprint $u_j'$,
$j \in [q].$ Consider the following sequence
$$S=((d_1,d'_1),\ldots,(d_1,d'_q),(d_2,d'_1),\ldots,(d_2,d'_q),\ldots,(d_p,d'_1),\ldots,(d_p,d'_q))$$
of vertices of $G \boxtimes H$. It is clear that $\widehat{S}$ is a
dominating set. Moreover a chosen vertex $(d_i,d_j')$ is legal since
it footprints $(u_i,u_j')$. Hence $\gamma_{gr}(G\boxtimes H)\ge
pq=\gamma_{gr}(G)\gamma_{gr}(H)$. \qed

We conjecture that the lower bound of Proposition~\ref{strong} is always tight:

\begin{conjecture}\label{con:strong}
For any graphs $G$ and $H$, $\gamma_{gr}(G\boxtimes H)= \gamma_{gr}(G)\gamma_{gr}(H).$
\end{conjecture}

Recall that the edge cover number $\theta_e(G)$ presents an upper bound for the $\ggr(G)$ in any graph $G$. We next show that this parameter behaves nicely in strong products of triangle-free graphs.

\begin{proposition}\label{Tfree}
If $G$ and $H$ are triangle-free graphs, then $\theta_e(G \boxtimes H) =|E(G)|\cdot |E(H)|.$
\end{proposition}
\proof Let $e=gg'$ and $f=hh'$ be arbitrary edges in $G$ and $H$,
respectively.  Then the edge $(g,h)(g'h')$ of $G \boxtimes H$ lies
in a unique maximal clique $Q_{e,f}$ induced by the four vertices in
$\{g,g'\}\times \{h,h'\}$. It follows that $\theta_e(G \boxtimes H)
\geq |E(G)|\cdot |E(H)|=\theta_e(G) \theta_e(H).$ On the other hand
the set of cliques $\{Q_{e,f}:\ e \in E(G), f \in E(H) \}$ forms an
edge clique cover of $G \boxtimes H$, so that $\theta_e(G \boxtimes
H) \leq |E(G)|\cdot |E(H)|.$ \qed

\begin{corollary}
If $k,l \geq 2,$ then $\ggr(P_k\boxtimes P_l)=(k-1)(l-1).$
\end{corollary}
\proof
The lower bound follows from Proposition~\ref{strong} and the upper bound from Proposition~\ref{prp:edge_cover} and Proposition~\ref{Tfree}.
\qed


We approach Conjecture~\ref{con:strong} with the following upper bound on the Grundy domination number of the strong product of graphs.

\begin{proposition} \label{Grundy-strong-lower}
Let $G$ and $H$ be arbitrary graphs. Then $$\ggr(G\boxtimes H) \leq \min\{|V(G)|\ggr(H),\ggr(G)|V(H)|\}.$$
\end{proposition}
\proof Let $D$ be a (Grundy) dominating sequence of $G\boxtimes H.$
Consider a layer $G^h$, $h \in V(H)$, and let $D^h$ be the
subsequence of $D$ that consists only of the vertices in $G^h.$ We
claim that the corresponding sequence $p_G(D^h)$ is a legal sequence
in $G$. Indeed, let $(g,h)\in D^h$ and let $(g',h')$ be a vertex
footprinted by $(g,h)$ with respect to the sequence $D$. It is clear
that $g$ footprints $g'$ with respect to the sequence $p_G(D^h).$
This implies that in each $G$-layer there are at most $\ggr(G)$
vertices from $D$. Therefore,  $\ggr(G\boxtimes H) \leq
\ggr(G)|V(H)|.$ By reversing the roles of $G$ and $H$ the claimed
inequality follows. \qed

Recall that a vertex is called simplicial if its neighborhood
induces a complete graph.

\begin{proposition} \label{G-v}
Let $G$ and $H$ be arbitrary graphs. If  $v$ is a simplicial vertex in
$G$, then
 $$\ggr(G\boxtimes H) \leq \ggr(H)+\ggr((G-v)\boxtimes H).$$
\end{proposition}
\proof Consider  a Grundy dominating sequence $D$ of $G \boxtimes H$
such that $D$ contains  maximum number of vertices from $^v\!H$. Let
$D^1$ be the subsequence of $D$ that consists only of the vertices
in $^v\!H$ and let $D^2$ be the complementary subsequence of $D$.
Analogously to the proof of Proposition~\ref{Grundy-strong-lower},
one can show that $D^1$ is a dominating sequence in $^v\!H$.
Consequently, $|\widehat{D^1}| \le \ggr(H)$.

Now assume that a vertex $(g,h)$, which is not from $^v\!H$,
footprints a vertex $(v,h')$ in $^v\!H$. Since $v$ is simplicial,
$N_G[v] \subseteq N_G[g]$. Further, either $h=h'$ or $(v,h)$ and
$(v,h')$ must be adjacent. Therefore, if $(g,h)$ is replaced by
$(v,h)$ in $D$, we obtain a Grundy dominating sequence again, and
this one contains more vertices from $^v\!H$ than $D$ did. This
contradicts  the choice of $D$ and proves that no vertex from
$\widehat{D^2}$ footprints a vertex outside  $(G-v)\boxtimes H$. We
may conclude that $D^2$ is a dominating sequence in
$(G-v)\boxtimes H$, and the desired inequality
$$\ggr(G\boxtimes H)=|\widehat{D^1}|+|\widehat{D^2}| \leq \ggr(H)+\ggr((G-v)\boxtimes H)$$
holds.
 \qed

We say that a graph $G$ \emph{satisfies
Conjecture~\ref{con:strong}} if for every graph $H$,
$\gamma_{gr}(G\boxtimes H)= \gamma_{gr}(G)\gamma_{gr}(H)$ holds. As
a consequence of Proposition~\ref{G-v}, we obtain the following
result related to the conjecture.
\begin{corollary} \label{caterpillar}
Every caterpillar satisfies Conjecture~\ref{con:strong}.
\end{corollary}
\proof
 Let $G$ be a caterpillar and let $H$ be an arbitrary graph. If  $G$
is of  order $1$ or $2$, $\gamma_{gr}(G\boxtimes H)=
\gamma_{gr}(G)\gamma_{gr}(H)= \gamma_{gr}(H)$ can be shown directly.
Then, we proceed by induction on $n$. As we mentioned at the end of
Section~\ref{sec:prelim}, it is shown in \cite{bgmrr-2014} that any
nontrivial caterpillar $G$ of order $n$ satisfies
$\gamma_{gr}(G)=n-1$. If $n \ge 3$, delete an arbitrary leaf $v$
from $G$. By the induction hypothesis we have
$\gamma_{gr}((G-v)\boxtimes H)= (n-2)\gamma_{gr}(H)$. Then,
Proposition~\ref{G-v} gives
 $$\ggr(G\boxtimes H) \leq \ggr(H)+(n-2)\gamma_{gr}(H)=(n-1)\gamma_{gr}(H)=\ggr(G)\ggr(H).$$
 Together with Proposition~\ref{strong} this establishes our
 statement.
 \qed

We immediately derive the exact value of Grundy domination number of cylinders and
an upper bound for tori.

\begin{corollary}
If $k\geq 2$ and $l \ge 3$, then $\ggr(P_k\boxtimes C_l) =
(k-1)(l-2)$.
\end{corollary}

\begin{corollary}
If $3 \leq k \leq l,$ then $\ggr(C_k\boxtimes C_l) \leq (k-2)(l-1)$.
\end{corollary}

Note that the conjectured value for a torus is  $\ggr(C_k\boxtimes
C_l) = (k-2)(l-2).$ We further remark that, by
Corollary~\ref{caterpillar},  strong products of any number of
caterpillars also satisfy Conjecture~\ref{con:strong}. Particularly,
the following exact results can be derived:
\begin{align*}
\ggr(P_{k_1}\boxtimes \cdots \boxtimes P_{k_n})&=(k_1-1)\cdots
(k_n-1)\\
\ggr(P_{k_1}\boxtimes \cdots \boxtimes P_{k_{n}}\boxtimes
C_{l})&=(k_1-1)\cdots (k_{n}-1)(l-2).
\end{align*}

Finally, we consider a  graph operation related to the conjecture.
We say that $G'$ is obtained from $G$ by substituting a vertex $v
\in V(G)$ with $K_\ell$ if $v$ is replaced with the complete graph on
$\ell$ vertices such that each of these $\ell$ new vertices is made
adjacent to the entire $N_G(v)$.

\begin{proposition} If $G$ satisfies Conjecture~\ref{con:strong}
and $G'$ is obtained from $G$ by  substituting a vertex $v \in V(G)$
with a complete graph, then $G'$ also satisfies
Conjecture~\ref{con:strong}.
\end{proposition}
\proof Clearly, a graph isomorphic to $G'$ can be obtained from $G$
by successively adding (true) twin vertices to $v$.
  It was observed in \cite{bgk2016+} that if $x$
has a twin in a graph $F$ then  $\ggr(F-x)=\ggr(F)$. Further, if $x$
has a twin in $F$, then for every graph $H$ and for every $h \in
V(H)$, the vertex $(x,h)$ also has a twin in $F \boxtimes H$. These
imply $\ggr((F-x)\boxtimes H)=\ggr(F \boxtimes H)$. Using these
statements in opposite direction and successively, we obtain the
following. If $G$ satisfies Conjecture~\ref{con:strong} then
$$ \ggr(G'\boxtimes H)=\ggr (G\boxtimes H)= \ggr
(G)\ggr(H)=\ggr(G')\ggr(H)$$ holds for every graph $H$. \qed

\section*{Acknowledgements}

Research of B.\ Bre\v sar, T.\ Gologranc, S.\ Klav\v zar, and G.\ Ko\v smrlj was supported by Slovenian Research Agency under the grants N1-0043 and  P1-0297.

Research of Cs.\ Bujt\' as, B.\ Patk\' os, Zs.\ Tuza, and M.\ Vizer
was supported by the National Research, Development and Innovation
Office -- NKFIH under the grant SNN 116095.

Research of B.\ Patk\'os was supported by the J\'anos Bolyai Research Fellowship of the Hungarian Academy of Sciences.

\baselineskip11pt

\end{document}